\documentclass{article}

\usepackage[preprint]{arxiv}

\makeatletter
\renewcommand{\@noticestring}{} % remove footer footnote text
\makeatother

\usepackage[utf8]{inputenc} % allow utf-8 input
\usepackage[T1]{fontenc}    % use 8-bit T1 fonts
\usepackage{hyperref}       % hyperlinks
\usepackage{url}            % simple URL typesetting
\usepackage{booktabs}       % professional-quality tables
\usepackage{amsfonts}       % blackboard math symbols
\usepackage{nicefrac}       % compact symbols for 1/2, etc.
\usepackage{microtype}      % microtypography
\usepackage{graphicx}
\usepackage{subcaption}
\usepackage{algorithm}
\usepackage{algpseudocode}

\usepackage{amsfonts, amsmath,amssymb, amsthm}
\newtheoremstyle{spacedstyle} 
  {10pt} {10pt} {\itshape} {} {\bfseries} {.} { } {}

\theoremstyle{spacedstyle}
\newtheorem{theorem}{Theorem}[section]
\newtheorem{lemma}[theorem]{Lemma}
\newtheorem{proposition}[theorem]{Proposition}

\newtheorem{definition}[theorem]{Definition}

\title{Greedy Selection under Independent Increments: A Toy Model Analysis}

\author{
  Huitao Yang \\
  Fudan University \\
  \texttt{htyang21@m.fudan.edu.cn} \\
}

\begin{document}
\maketitle

\begin{abstract}
We study an iterative selection problem over $N$ i.i.d. discrete-time stochastic processes with independent increments. At each stage, a fixed number of processes are retained based on their observed values. Under this simple model, we prove that the optimal strategy for selecting the final maximum-value process is to apply greedy selection at each stage. While the result relies on strong independence assumptions, it offers a clean justification for greedy heuristics in multi-stage elimination settings and may serve as a toy example for understanding related algorithms in high-dimensional applications.
\end{abstract}

\section{Related Work}

Greedy strategies have been extensively analyzed in combinatorial optimization and submodular maximization, where they often offer provable approximation guarantees \cite{nemhauser1978analysis}. In the context of stochastic processes, optimal stopping problems and secretary-style selection problems also study decision-making with limited information \cite{ferguson1989solved}, although they typically involve single-item selection rather than staged elimination.

This problem bears conceptual resemblance to algorithmic strategies in large language model (LLM) alignment. Techniques such as \emph{Best-of-N sampling} and \emph{speculative decoding} \cite{bestofn,lee2023speculative} involve generating and filtering multiple candidate trajectories with partial information, aiming to retain high-quality outputs. \emph{Speculative Rejection}\cite{sun2024fastbestofndecodingspeculative} is especially related where responses are iteratively eliminated based on their partial scores. While these strategies typically operate in high-dimensional latent spaces, they implicitly rely on the same idea of preserving top candidates across stages. Our result provides theoretical grounding for greedy selection in simplified models, which may inform the analysis of these heuristics under stochastic process assumptions.

\section{Problem Setting}

We consider $N$ i.i.d. discrete-time stochastic processes 
\[
X_1(t), \ldots, X_N(t), \quad t = 0, 1, \ldots, T,
\]
with $X_i(0) = 0$ and independent increments. That is, for each $i$, the increments $X_i(t+1) - X_i(t)$ are independent across $t$.

Our goal is to identify the index
\[
i_* = \arg\max_{i \in [N]} X_i(T),
\]
without observing full trajectories. Instead, selection is done iteratively by filtering subsets of processes at intermediate times. We formalize this below.

\begin{definition}[Iterative Selection Algorithm]
Let $0 < t_1 < \cdots < t_k = T$ be observation times, and $N > n_1 > \cdots > n_k = 1$ be subset sizes. For fixed realization $\omega \in \Omega$,  an \emph{iterative selection algorithm} consists of $k$ maps $\{f_j\}_{j=1}^k$ such that:
\begin{itemize}
  \item $f_1$ selects $n_1$ indices from all observed histories up to $t_1$:
  \[ f_1: \mathcal{H}_1: =  \{X_i(t)\}_{i\in[N],\, t\le t_1} \mapsto \Sigma_1 \subset [N], \quad |\Sigma_1| = n_1. \]

  \item For $j > 1$, $f_j$ selects $n_j$ indices from updated histories over surviving indices:
  \[ f_j:  \mathcal{H}_i\mapsto \Sigma_j \subset \Sigma_{j-1}, \quad \mathcal{H}_i = \mathcal{H}_{(i-1)} \cup \{X_i(t)\}_{i\in\Sigma_{j-1}, t_{j-1}<t\le t_j},  \quad |\Sigma_j| = n_j. \]

  The final output is the single element $i_* \in \Sigma_k$.
\end{itemize}
\end{definition}

We denote an iterative selection algorithm as a map 
  \[\text{Alg} : \{X_n(t)\}_{n \in \mathbb{N},  t \leq T} \in \mathbb{R}^{N \times T} \mapsto X_{i^*}(T) \in \mathbb{R}, \]

  though $\text{Alg}$ only relies on a path subspace with shape aligned with $\mathcal{H}_k$.  

\section{Greedy Selection and Alignment}

\begin{definition}[Greedy Selection Algorithm]
The \emph{greedy algorithm} selects the top $n_j$ processes by their observed values at each $t_j$ among the currently retained set. Denote these selected sets as $\Sigma_j^*$.
Mathematically, 
\[\Sigma_j^* = \{n \in \Sigma^*_{(j-1)} ; X_n(t_j) \text{ is among the highest } n_j \}.\]
\end{definition}

\begin{definition}[Temporal Index System]
Fix an iterative selection algorithm and consider a single realization (sample path) of all processes. Define:
\begin{itemize}
  \item $\Sigma_j$ as the selected set at time $t_j$;
  \item $W_j := \Sigma_j$ and $L_j := [N] \setminus W_j$ for all $j$.
\end{itemize}

At time $t_1$, assign indices to $X_n$ based on their ranks among all $X_n(t_1)$. At each later time $t_j$, rank the processes in $W_{j-1}$ based on $X_n(t_j)$ and assign indexes of $1,\ldots, n_{j-1}$. Processes in $L_{j-1}$ keep their previous indices.
\end{definition}

\begin{definition}[Greedy Index System]
The \emph{greedy index system} is the temporal index system induced by the greedy selection algorithm.
\end{definition}

\begin{definition}[Alignment Map]\label{def:  map}
Given any iterative selection algorithm and a realization $\omega$ of the processes, define a bijective map $\phi$ over the space of full process paths:

Let $Y_n$ denote the image processes under $\phi$. At $t_1$, define $Y_n^{\leq t_1} := X_n^{\leq t_1}$ for all $n$.

For each $t_i$, define the increments $Y_m(t) - Y_m(t_{i-1})$ over $(t_{i-1}, t_i]$ to match $X_n(t) - X_n(t_{i-1})$ if $Y_m$'s greedy index at $t_i$ equals $X_n$'s temporal index under the original algorithm.
\end{definition}

\begin{proposition}[Measure Preservation]\label{prop:measure-preservation}
The alignment map $\phi$ is an almost-everywhere differentiable bijection with Jacobian determinant 1. Thus, it preserves the joint distribution of $\{X_i\}_{i=1}^N$ under the assumptions of i.i.d. processes with independent increments.
\end{proposition}

\begin{figure}[H]
    \centering
    \begin{subfigure}[b]{0.45\textwidth}
        \centering
        \includegraphics[width=\textwidth]{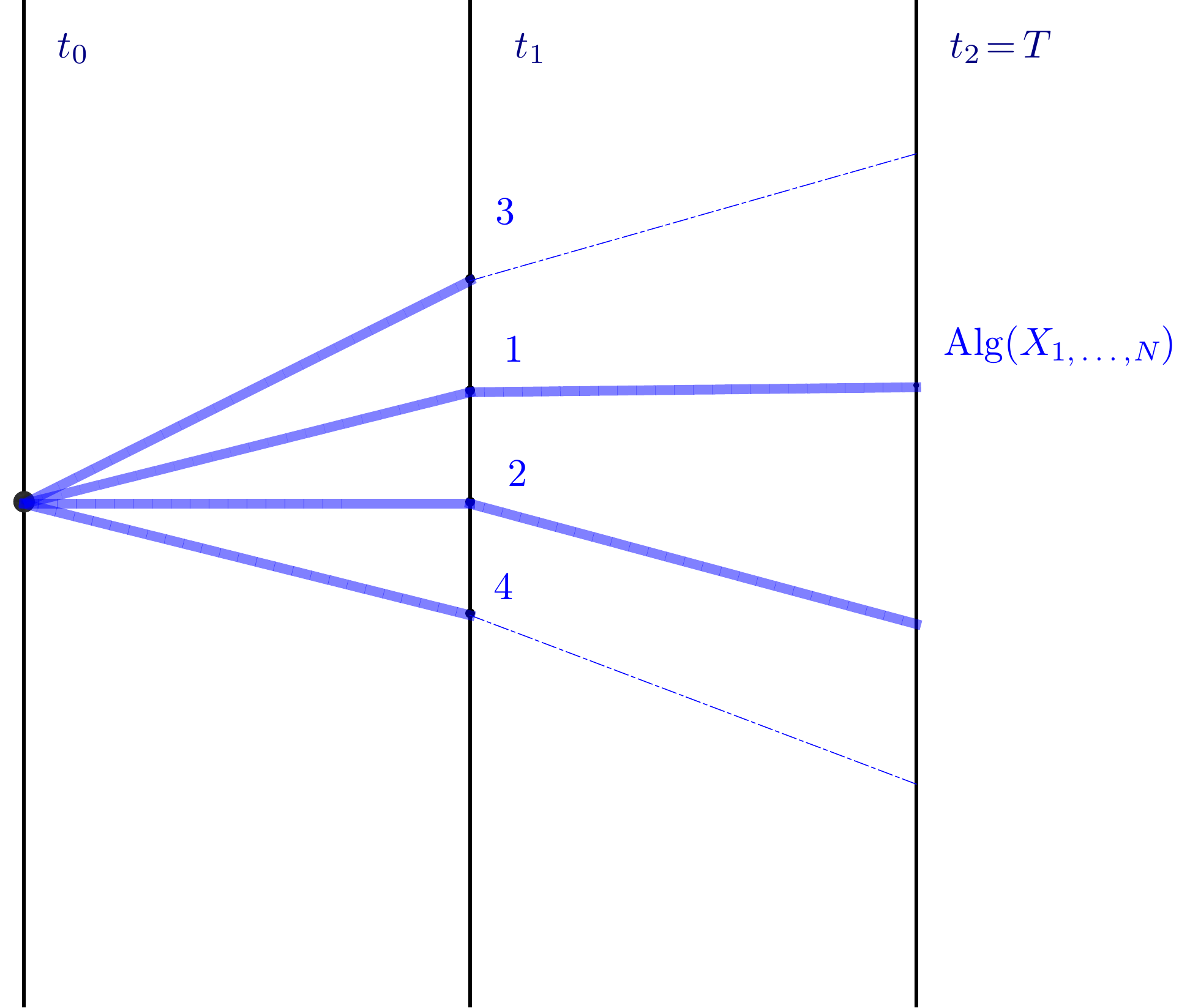}
        \caption{$X_{1,2,3,4}$}
        \label{fig:gamma}
    \end{subfigure}
    \hfill
    \begin{subfigure}[b]{0.45\textwidth}
        \centering
        \includegraphics[width=\textwidth]{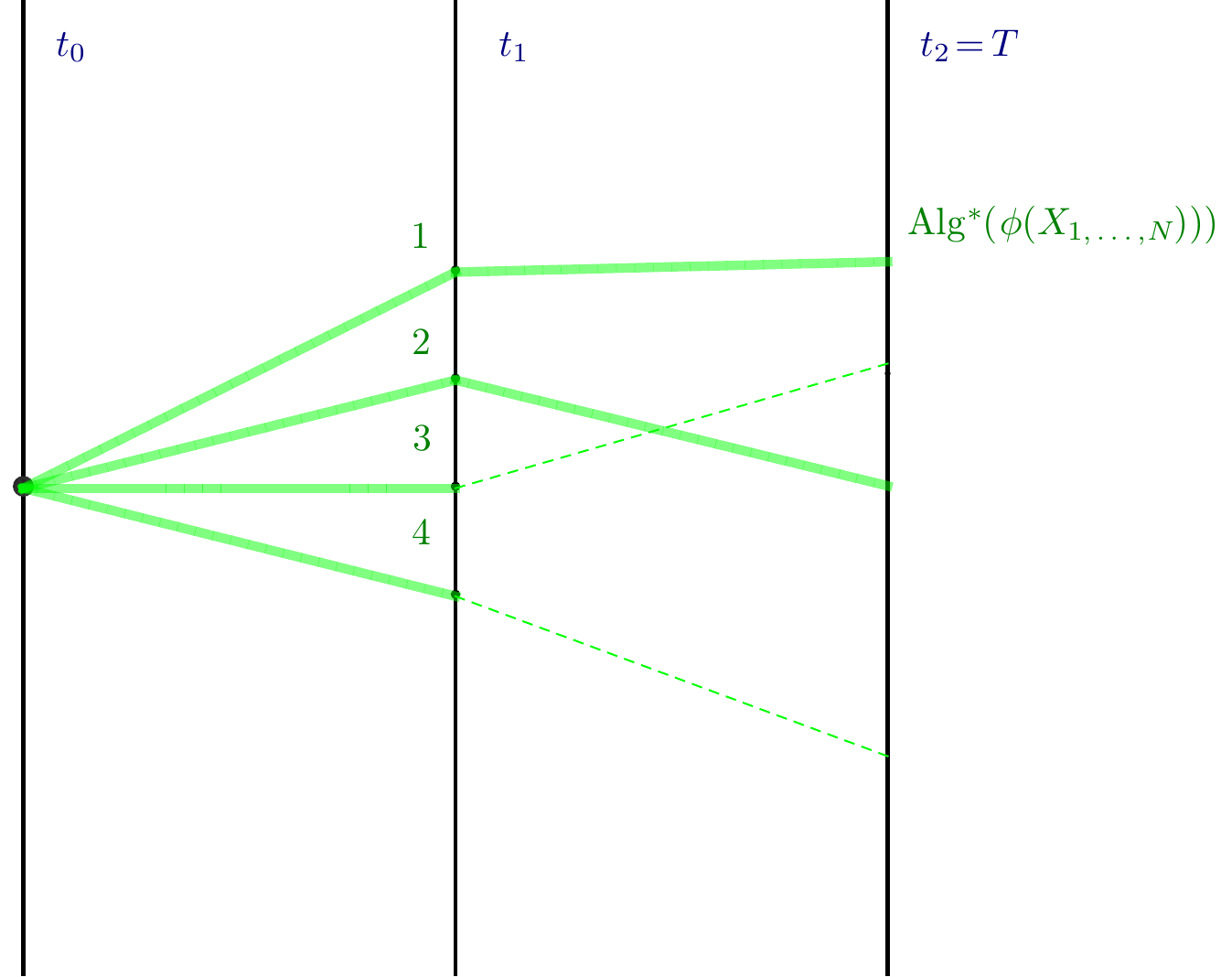}
        \caption{$Y_{1,2,3,4}$}
        \label{fig:gamma_bar}
    \end{subfigure}
    \caption{An illustrative example is provided for Proposition~\ref{prop:greedy-dominance}. Connected segments of the same color (blue or green) represent the selected processes, while dotted lines indicate rejected ones. The numbers denote the temporal indexes for $\text{Alg}$ and $\text{Alg}^*$ respectively.}
    \label{fig:appendix}
\end{figure}

\begin{proposition}[Greedy Dominance under Alignment]\label{prop:greedy-dominance}
Let $\text{Alg}$ be any iterative selection algorithm and $\text{Alg}^*$ the greedy algorithm. Then for any realization $\omega$:
\[\text{Alg}(X_1, \ldots, X_N)(\omega) \le \text{Alg}^*(\phi(X_1, \ldots, X_N))(\omega).\]
\end{proposition}

\section{Main Result}

\begin{theorem}
For i.i.d. discrete-time processes with independent increments, the greedy selection algorithm maximizes the expected value of the final selected process:
\[\mathbb{E}[\text{Alg}(X_1, \ldots, X_N)] \le \mathbb{E}[\text{Alg}^*(X_1, \ldots, X_N)].\]
\end{theorem}

\begin{proof}
Let $\phi$ be the alignment map. We prove the statement via:
\[\mathbb{E}[\text{Alg}(X)] \le \mathbb{E}[\text{Alg}^*(\phi(X))] = \mathbb{E}[\text{Alg}^*(X)].\]
where the first inequality follows from Proposition~\ref{prop:greedy-dominance} and the last equality follows from the change of variable theorem and Proposition~\ref{prop:measure-preservation}.
\end{proof}

\section{Discussion}

Our main result shows that under the assumption of independent increments, the optimal iterative selection strategy is greedy—at each stage, simply retaining the top-valued processes based on their current observations suffices. Notably, the optimal algorithm does not exploit any past observations beyond the current timestep. This striking simplicity arises from the strong structural assumption of independent increments: the future evolution of each process is conditionally independent of its past, given its present value.

If we weaken this assumption—for example, allowing for temporal dependencies or non-stationary dynamics—the optimal selection strategy will, in general, require incorporating historical observations. In such settings, an ideal algorithm would compute posterior estimates of future values using both the observed history and some probabilistic prior over the process dynamics. This would depart significantly from the greedy structure and may involve Bayesian inference or learning-based prediction.

This perspective also suggests an intriguing connection to recent trends in large language model (LLM) alignment. If a reward model guiding generation is trained in such a way that its partial outputs behave like processes with approximately independent increments—e.g., scores evolve in a memoryless or marginally stable fashion—then selection heuristics such as speculative decoding and speculative rejection can be expected to perform well. In particular, speculative rejection operates by iteratively pruning candidate outputs using partial reward signals, which parallels the greedy algorithm described here. Understanding when such reward models induce effective independence properties may shed light on why these heuristics work and how to improve them.

\bibliographystyle{plain}
\bibliography{refs}

\appendix
\section{Omitted Proofs and Comments}

\subsection{Proof for Proposition~\ref{prop:measure-preservation}}
Consider a subspace of the product space of $N$ paths where all processes have distinct values at $t_j$, which is $0$-measure different from the whole path space. Each point on this space has a neighborhood where all rankings are preserved and the alignment is locally a linear permutation map. Therefore the map is differentiable on this subspace. Same arguments work for the inverse. Thus we prove the statement about differentiability. 

Denote the linear transformation from product of $N$ paths to its corresponding increments as $\Phi$, or specifically:
\[
\Phi: \{X_{n}(t)\}_{n \in [N], t\in [0, T]} \mapsto \bigcup_j \{X_{n}(t) - X_n(t_{(j-1)}))\}_{n \in [N], t\in [t_{(j-1)}, t_j)} 
\]
The defined alignment map locally operates a permutation $\sigma$ on the image space of $\Phi$, thus with determinant 1. Thus the overall jacobian is:
\[
\det (\Phi^{-1} \sigma \Phi) =\det (\sigma) = 1
\]

\subsection{Proof for Proposition~\ref{prop:greedy-dominance}}
We prove a stronger version of proposition~\ref{prop:greedy-dominance} as follows:
\begin{proposition}

    If we denote $Y_{1,\ldots, N} = \phi(X_{1,\ldots, N})$, then at each $t_j$, process pair $(X_n, Y_m)$ such that $Y_m$ has the greedy index same as $X_n$'s temporal index, $Y_m(t_j)\geq X_n(t_j)$.
\end{proposition}
We prove the above proposition through induction. 

At initialization, the process pairs are $(X_n, Y_n)$. $X_n(t_1) = Y_n(t_1)$ by definition~\ref{def: map}. 

At $t_j$ with $j>1$, assume that the statement is true for previous rounds. Denote $\Sigma_{j-1}$ and $\Sigma^*_{j-1}$ the selected processes' index set at $t_{(j-1)}$ (by $\text{Alg}$ and $\text{Alg}^*$). Note that for each process pair $(X_n, Y_m)$ at $t_{j-1}$:
\[
X_n(t_j) - X_n(t_{(j-1)}) = Y_m(t_j) - Y_m(t_{(j-1)}). 
\]
Thus $X_n(t_j) \geq Y_m(t_j)$. Although $(X_n, Y_m)$ may no longer be a process pair at $t_j$, we still deduce the statement because the highest $n_j$ from $Y_m(t_{(j-1)})$s must be higher than the highest $n_j$ from $X_n(t_{(j-1)})$s accordingly. This is an example of the following lemma:
\begin{lemma}
For three vectors $A, B, C\in \mathbb{R}^k$ with $A_i \leq B_i$, then the $m$-th order statistics satisfies: 
\[
(A + C)(m) \leq (B + C)(m)
\]
for any $m \in [k]$.
\end{lemma}
\begin{proof}
    Suppose that there is $m$ such that $(A + C)(m) > (B + C)(m)$. Denote $J$ as the following index set:
    \[
    J = \{ i \in [k]; (A + C)_j \geq (A + C)(m)\},  
    \]
    then $|J| \geq m$.  And for each $i\in J$, 
    \[
    (B + C)(m)<(A + C)(m) \leq (A + C)_i  \leq (B + C)_i
    \]
    thus there is $|J|\geq m$ elements of $(B + C)$ strictly greater than $(B + C)(m)$, which yields a contradiction. 
\end{proof}

Now since the stronger proposition holds, we have proved the original proposition. 
\end{document}